\documentclass[12pt,oneside,reqno]{amsart}
\usepackage[all]{xy}
\usepackage{amsfonts,amsmath,oldgerm,amssymb,amscd}
\UseComputerModernTips
\numberwithin{equation}{section}
\usepackage[breaklinks]{hyperref}
\allowdisplaybreaks

\usepackage{enumerate}
\usepackage{caption} 
\newcommand{\Zstroke}{%
  \text{\ooalign{\hidewidth\raisebox{0.2ex}{--}\hidewidth\cr$Z$\cr}}%
}
\newtheorem{theorem}{Theorem}[section]
\newtheorem{proposition}[subsection]{\bf Proposition}

\newtheorem{lemma}[subsection]{{\bf Lemma}}

\newcommand{\m}{\mbox{$\mathfrak m$}}

\begin{document}

	\title[Resolution of the Skolem Problem for $k$-Generalized Lucas Sequences
	]{Resolution of the Skolem Problem for $k$-Generalized Lucas Sequences
	} %\\ \today}

\author[M. Mohapatra]{M. Mohapatra}
\address{Monalisa Mohapatra, Department of Mathematics, National Institute of Technology Rourkela, Odisha-769 008, India}
\email{mmahapatra0212@gmail.com}

\author[P. K. Bhoi]{P. K.  Bhoi}
\address{Pritam Kumar Bhoi, Institute of Mathematics and Applications, Andharua, Bhubaneswar, Odisha -751 029, India}
\email{pritam.bhoi@gmail.com}

\author[G. K. Panda]{G. K. Panda}
\address{Gopal Krishna Panda, Department of Mathematics, National Institute of Technology Rourkela, Odisha-769 008, India}
\email{gkpanda\_nit@rediffmail.com}

\thanks{2020 Mathematics Subject Classification: Primary 11B39, Secondary 11J86, 11D61.\\
	Keywords: Zero multiplicity, Skolem Problem, $k$-Generalized Lucas numbers, Baker-Davenport reduction method, linear forms in logarithms,}

\begin{abstract}
	This paper provides a complete solution to Skolem’s problem for the $k$-generalized Lucas sequence $(L_n^{(k)})_{n \in \mathbb{Z}}$ with a primary focus on its behavior at negative indices. We characterize the zero-distribution of this sequence by identifying and bounding all indices $n < 0$ such that $L_n^{(k)} = 0$. Our central result establishes that the zero-multiplicity $\delta_k$ of the sequence is $(k-1)(k-2)/2$
for all $k.$
\end{abstract}

\maketitle
\pagenumbering{arabic}
\pagestyle{headings}
 \section{Introduction}
The $k$-generalized Lucas sequence $(L_n^{(k)})_{n\geq2-k}$ is a multi-order extension of the standard Lucas sequence for any integer $k \geq 2$. It is governed by the $k$-th order additive recurrence $$L_n^{(k)} = \sum_{j=1}^{k} L_{n-j}^{(k)}, \quad \text{for } n \geq 2,$$ with initial terms $L_0^{(k)} = 2$, $L_1^{(k)} = 1$, and $L_n^{(k)} = 0$ for indices $-(k-2) \leq n \leq -1$. Notably, the case where $k=2$ reverts to the classical Lucas numbers.
The behavior of this sequence is linked to its characteristic polynomial $$\Psi_{k}(x) = x^{k} - 2x^{k-1} - \cdots - x - 1.$$ Research by Wolfarm \cite{Wolfarm} showed that $\Psi_{k}(x)$ is irreducible over the rationals. It has one dominant positive real root $\gamma(k)$ found in the interval $(2(1-2^{-k}), 2)$. The other roots $ \gamma_2, \gamma_3, \ldots, \gamma_k $ are inside the unit circle and satisfy the condition $3^{1/k} < |\gamma_i| <1 $. Let $\gamma = \gamma_1, \gamma_2, \ldots, \gamma_k$ represent the roots of the polynomial $\Psi_k(z)$. They are labeled so that \begin{equation}\label{eq:rootsorder} \gamma > |\gamma_2| \geq |\gamma_3| \geq \cdots \geq |\gamma_{k-1}| \geq |\gamma_k|, \quad \text{where } \gamma=\gamma(k). \end{equation}
For $k \geq 2$, let
\begin{equation*}
	f_{k}(x)=\frac{x-1}{2+(k+1) (x-2)}. 
\end{equation*}
In \cite{DDu2014},  Dresden and Du proved that
\begin{equation}\label{eq2} \quad\left|f_{k}\left(\gamma_{i}\right)\right|<1,\quad2 \leq i \leq k.
\end{equation}
So, the number $f_{k}(\gamma)$ is not an algebraic integer. 
In addition, they showed that
\begin{equation}\label{binet like}
	L_{n}^{(k)}=\sum_{i=1}^{k} (2\gamma_i-1)f_{k}\left(\gamma_{i}\right) \gamma_{i}^{{n-1}} \quad \text { and } \quad\left|L_{n}^{(k)}-f_{k}(\gamma) (2\gamma-1)\gamma^{n-1}\right|<\frac{3}{2} 
\end{equation}
hold for all $n \geq 2-k$ and $k \geq 2$. In \cite{rl2016}, it was shown that the logarithmic height of $f_k(\gamma)$ satisfies
\begin{equation}\label{eq8}  h(f_k(\gamma)) < 2 \log k \quad \text{for all } k \geq 2.  \end{equation}

A persistent challenge in number theory is determining if and when a linear recurrence sequence hits zero. This is formally known as the Skolem problem, which seeks to identify the set of indices $\mathcal{Z}(u) = \{n \in \mathbb{Z} : u_n = 0\}$. The size of this set is termed the sequence's zero-multiplicity. Calculating the cardinality of $\mathcal{Z}(u)$ is not currently supported by any general methodology. Nevertheless, Skolem provided a fundamental breakthrough regarding this issue with the following theorem:

\emph{{\bf Skolem Theorem}\cite{ShoreyTijdeman2008, Skolem1934}{\bf:} If the coefficients of the linear recurrence sequence are rational, then the set $\Zstroke(u)$ is a union of finitely many arithmetic progressions together with a finite set.}

Extensions of this result were obtained by Mahler, Lech, and others \cite{Lech1953, Mahler1935, MahlerCassels1956}. They proved that for linear recurrence sequences with rational or algebraic coefficients, the zero set is a finite union of isolated points and arithmetic progressions. Their arguments, however, were ineffective. Berstel and Mignotte \cite{BerstelMignotte1976} later developed methods to find these arithmetic progressions.

In the study of linear recurrence sequences, a sequence is classified as \emph{degenerate} if the ratio of any two distinct roots of its characteristic polynomial is a root of unity. Otherwise, the sequence is \emph{non-degenerate}.
For non-degenerate linear recurrence sequences, determining the exact zero-multiplicity remains a central problem. Some partial progress have been achieved for sequences with integer coefficients and small order as discussed in \cite{Berstel1974, Beukers1991, BeukersTijdeman1984, Ward1959}. In contrast, for any order $k \ge 4$, only very large bounds in terms of $k$ are available; for instance, Wolfgang M. Schmidt \cite{Schmidt1999} established
$\Zstroke(u) \leq \exp(\exp(\exp(3k\log k))).$ Hagedorn \cite{Hagedorn2005}  showed that if all roots of the characteristic polynomial are real, the zero-multiplicity is at most $2k-3$. Recent literature has seen the Skolem problem solved for specific families, including $k$-generalized Fibonacci numbers \cite{ggl2020} and $k$-generalized Pell numbers \cite{MBP2025}. 
In this research, we investigate the zero-multiplicity of the $k$-generalized Lucas sequence $(L_n^{(k)})_{n \in \mathbb{Z}}$. To facilitate this, we extend the sequence to negative indices $n$ using the relation
\begin{equation}\label{negative representation}
L_{n-k}^{(k)} = -L_{n-(k-1)}^{(k)} - L_{n-(k-2)}^{(k)} - \cdots -L_{n-2}^{(k)}- L_{n-1}^{(k)} + L_n^{(k)}.
\end{equation}
Furthermore, we consider the sequence $(Q_n^{(k)})_{n \in \mathbb{Z}}$ defined for $n \geq 0$ with the specific initial conditions $Q_i^{(k)} = 0$ for $1 \leq i \leq k-2$ and $Q_{k-1}^{(k)} = -1$. For all $n \geq k$, this sequence satisfies the recursive relation \begin{equation}\label{Q_n}Q_n^{(k)}=Q_{n-k}^{(k)}-
 Q_{n-(k-1)}^{(k)}-\cdots-Q_{n-1}^{(k)}, \hspace{0.1cm}\text{holds for all } n \geq k.\end{equation}
The following observations concern the zero-multiplicity of this sequence for non-positive indices $n \leq 0$:
\begin{table}[h!]
\centering
\begin{tabular}{c|l|c}

$k$ & Indices  & Zero-Multiplicity \\ 
\hline 2 & -- & 0 \\
3 & $-1$ &1\\
4 & $[-2, -1],-6$ &3\\
5 & $[-3, -1], [-8,-7], -13$ &6\\
6 &  $[-4, -1], [-10, -8]$, $[-16,-15]$, $-22$ & 10 \\
7 &  $[-5, -1], [-12, -9], [-19,-17]$, $[-26,-25]$, $-33$ & 15\\
\end{tabular}

\label{tab:zero_multiplicity}
\end{table}
\begin{theorem}\label{main result 1}  
The largest nonnegative integer solution $n$ to the Diophantine equation $L_{-n}^{(k)} = 0$ satisfies the following bounds:
\begin{itemize}
    \item[(i)] If $k$ is even, then $n < 2k^{k^2}\log(490k^2).$ 
    \item[(ii)] If $k$ is odd, then $n < 1.5 \cdot 10^{17} \cdot 1.454^{k^3} \cdot k^{12} \cdot (\log k)^2.$ 
\end{itemize}
\end{theorem} 
The explicit form of the terms can be expressed by introducing an auxiliary function $\psi(y, z)$ given by the sum of two binomial coefficients:\begin{equation}\label{psi}\psi(y,z) = \binom{y}{z} + \binom{y+1}{z+1}.\end{equation}We adopt the standard conventions that $\psi(y, -1) = 1$ and that the function vanishes in cases where $z > y \geq 0$.
\begin{theorem}[]\label{thm2}
If $k \ge 2$, then\begin{itemize}
    \item[(i)] \vspace{0.1cm} $Q_{mk+r}^{(k)} = 0$, for $0 \le m \le r \le k-2$,
    \item[(ii)]\vspace{0.1cm}$Q_{mk-1}^{(k)} = -2^{m-1},$ for $m \in [1, k-2],$
    \item[(iii)]\vspace{0.1cm} $Q_{mk+r}^{(k)} = (-1)^r \left[ \psi(m-1, r-1) 2^{m-r} + \psi(m-1, r) 2^{m-r-2} \right]$,\vspace{0.2cm}\\ for $0 \le r \le m \le k-2$,\vspace{0.1cm}
    \item[(iv)]for $r \in [-1, k-2]$ and $m \ge k-2$,\begin{equation*}\begin{split}Q_{mk+r}^{(k)} = \sum_{i=0}^{l} (-1)^{ik+r} &[ \psi(m-i-1, ik+r-1) 2^{m-i(k+1)-r} \\ +& \psi(m-i-1, ik+r) 2^{m-i(k+1)-r-2} ],\end{split}\end{equation*}  where $l=m$ for $k=2$ and $l = \lfloor (m+1)/(k-1) \rfloor$ if $k>2$. 
\end{itemize}
\end{theorem}

\begin{theorem}\label{coro 1.2} Let $(\delta_k)_{k \geq 2}$ be the zero-multiplicity counts of the $k$-generalized Lucas sequence. Then $\delta_2 = 0$, $\delta_3 = 1$ and 
$\delta_k =(k-2)(k-1)/2,$ for all $k.$ 
\end{theorem}
\section{Auxiliary Results}To resolve Diophantine equations efficiently, one often relies on Baker-type lower bounds for non-vanishing linear forms in logarithms. This section outlines the essential principles of algebraic number theory and the specific theorems used to compute and refine these bounds.

The complexity of an algebraic number $\theta$ is quantified by its absolute logarithmic height, denoted as $h(\theta)$. Let $\theta$ have a minimal polynomial $$g(X) = c_0 \prod_{i=1}^{k} (X - \theta^{(i)}) \in \mathbb{Z}[X],$$where $c_0 > 0$ and $\theta^{(i)}$ represent the conjugates of $\theta$. The logarithmic height of $\theta$ is defined as $$h(\theta) = \frac{1}{k} \left( \log c_0 + \sum_{j=1}^{k} \max\{0, \log |\theta^{(j)}|\} \right).$$ If $\theta = c/d$, $(d\neq 0)$ is a rational number with $c$ and $d$ are relatively prime, then $h(\theta)=\log$(max$\{|c|,d\}).$ 

Let $\theta$ and $\beta$ be any two algebraic numbers. The following properties of the logarithmic height $h$ are established in \cite[Theorem B5]{b2018}:
\begin{equation*}
	\begin{split}
		& (i)\hspace{0.2cm} h(\theta\pm\beta) \leq h(\theta)+h(\beta)+\log 2,\\
		&(ii)\hspace{0.2cm} h(\theta\beta^{\pm 1})\leq h(\theta)+h(\beta),\\
		&(iii)\hspace{0.2cm} h(\theta ^k)=|k|h(\theta).
	\end{split}
\end{equation*}

Building upon these notations, we cite a specialized result from \cite{bgl2016}, which refines the landmark theorem by Matveev \cite{m2000} and its later adaptations by Bugeaud et al. \cite{bms2006}. This theorem is instrumental in establishing the explicit upper bounds required for the variables in our primary equations.

\begin{theorem}\label{thm1}\cite{bgl2016} Let $\theta_1,\ldots,\theta_t$ be positive real numbers in an algebraic number field $\mathbb{K}$ of degree $d_{\mathbb{K}}$ and $e_1, \ldots, e_t$ be nonzero integers. If $\Lambda = \prod\limits_{i=1}^{t} \theta_{i}^{e_i} -1$ is not zero, then 
	\begin{equation}
	 \log |\Lambda| > -3 \cdot 30^{t+4} (t + 1)^{5.5}  d_{\mathbb{K}}^2 (1 + \log d_{\mathbb{K}})(1 + \log tB)A_1 \cdots A_t,
	\end{equation}
	where $B\geq \max\{|e_1|,\ldots,|e_t|\}$ and  $A_1, \ldots, A_t$ are positive integers such that $A_j \geq h'(\theta_j) = \max\{d_{\mathbb{K}}h(\theta_j),|\log \theta_j|, 0.16\},$ for $j= 1,\ldots,t$.
\end{theorem}
Next, we recall a refined version of the classical Baker--Davenport reduction lemma, due to Dujella and Peth\H{o} \cite{dp1998}. This lemma will be used to reduce the upper bounds on our variables.
\begin{lemma}\label{lem1}\cite{dp1998} Let $M$ be a positive integer and $p/q$ denote a convergent of the continued fraction of the real number $\tau$ such that $q > 6M$. Consider the real numbers $A, B, \mu$ with $A > 0$ and $B > 1$. Let $\epsilon = \left\lVert \mu q\right\rVert - M\left\lVert \tau q\right\rVert$, where $\left\lVert .\right\rVert$ denotes the distance from the nearest integer. If $\epsilon > 0$, then there exists no solution to the inequality
	\begin{equation*}
		0 <|u\tau - v +\mu|<AB^{-w},
	\end{equation*}
	in positive integers $u,v,w$ with $u\leq M$ and $w\geq \dfrac{\log(Aq/\epsilon)}{\log B}$.
\end{lemma}
\begin{lemma}\label{lem2}\cite{sl2014} Let $r\geq 1$ and $H > 0$ be such that $H > (4r^2)^r$ and $H>L/(\log L)^r$. Then $L < 2^rH(\log H)^r$.
\end{lemma}

\begin{lemma} \cite{ggl2020}\label{lem dege}
For a non-degenerate linear recurrence sequence, the number of zero terms (i.e., its zero-multiplicity) is always finite.
\end{lemma} 

Research by Mignotte \cite{Mignotte1984} confirmed that for the $k$-generalized Lucas sequence, the ratio $\gamma_i / \gamma_j$ is never a root of unity for any distinct indices $1 \leq i < j \leq k$. Consequently, this sequence is strictly non-degenerate, and by Lemma \ref{lem dege}, it possesses only a finite number of zeros.

The following lemma describe the distribution of the roots $\gamma_i$ and the behavior of the function $f_k(\gamma_i)$.
\begin{lemma}\label{lem:properties}\cite{ggl2020}
Let $\gamma = \gamma_1, \gamma_2, \ldots, \gamma_k$ be the roots of $\Psi_k(z)$ ordered by their magnitudes as defined in \eqref{eq:rootsorder}.

\begin{itemize}
    \item[(i)] If $\gamma_i$ and $\gamma_j$ are two roots satisfying $|\gamma_i| > |\gamma_j|$, then
    ${|\gamma_i|}/{|\gamma_j|} > 1 + 1.454^{-k^3}.$ 
    \item[(ii)] $1/2 \leq f_k(\gamma) \leq 3/4$, for all $k\geq 2$. 
    \item[(iii)] 
    $
    |f_k(\gamma_i)| < \min\left\{{1}/{2}, {2}/{(k-1)}\right\},$ for all $k\geq 2$ and $2\leq i\leq k$.
    \item[(iv)] For all $k\geq 2$,
    \begin{equation*}
    |\gamma_k| < 1 - \dfrac{\log\gamma}{2k}
    \quad \text{and} \quad
    |f_k(\gamma_k)| > \dfrac{\log\gamma}{2k(3k+1)}.
    \end{equation*}
    \item[(v)] For all $k\geq 2$ and $j\neq 1,$ we have
    \begin{equation*}
    |\gamma_k| < 1 - \dfrac{1}{2^8k^3}
    \quad \text{and} \quad
    |f_k(\gamma_k)| > \dfrac{1}{2^8k^3(3k+1)}.
    \end{equation*}
    \item[(vi)] If $k$ is even, then
 ${|\gamma_{k-1}|}/{|\gamma_k|} > 1 + ({1}/{k^{k^2}}).$
\end{itemize}
\end{lemma}
\begin{lemma}\label{lem2,2}\cite{ggl2024}
   Let $k \ge 2$ and $H_n^{(k)} = F_{-n}^{(k)}$. 
\begin{itemize}
    \item[(a)] $H_{mk-1} = 2^{m-1}$, for $m \in [1, k-1]$.
    \item[(b)]$H_{mk} = -(m+1)2^{m-2}$, for $m \in [1, k-1]$.
    \item[(c)] For $1 \le r < m \le k-1$,
    $$H_{mk+r} = -\sum_{j=r}^{m-1} 2^{m-1-j} H_{jk+r-1}.$$
\end{itemize} 
\end{lemma}
\begin{lemma}\cite{ggl2022}\label{lem6}
The inequality \begin{equation*}\frac{|\gamma_i|}{|\gamma_j|} > 1 + \frac{1}{10k^{9.6}(\pi/e)^k}\end{equation*}holds whenever $|\gamma_i| > |\gamma_j|$, for all $k \ge 4$.
\end{lemma}
\begin{lemma}\label{lemma_associate}
For $n\in \mathbb{Z}$,
\begin{equation}\label{associate_identity}
Q_n^{(k)} = 2H_{n-1}^{(k)} - H_n^{(k)}.
\end{equation}
\end{lemma}
\textbf{Proof:}
    From \cite{anl2023}, we get the identity $L_n^{(k)}=2F_{n+1}^{(k)}-F_n^{(k)}.$ Replacement of $n$ by $-n$ yields $L_{-n}^{(k)}=2F_{-(n-1)}^{(k)}-F_{-n}^{(k)}.$ As $L_{-n}^{(k)}=Q_n^{(k)}$ and $F_{-n}^{(k)}=H_n^{(k)}$, so we get \eqref{associate_identity}.

\section{Proof of Theorem \ref{main result 1}}

Consider the case where $k \ge 4$. We seek to solve the Diophantine equation $L_{-n}^{(k)} = 0$. By applying the Binet-like expansion from \eqref{binet like}, the problem is equivalent to determining all non-negative integers $n$ that satisfy the exponential identity 
\begin{equation}\label{eqn1}
\sum_{i=1}^{k} (2\gamma_i-1)f_k(\gamma_i)\gamma_i^{-(n+1)} = 0
\end{equation}
for integers $ n\geq 0 $,
where $ \gamma_1, \ldots, \gamma_k $ are the roots of the characteristic polynomial $ \Psi_k(z) $, ordered as in \eqref{eq:rootsorder}.
\vspace{-0.7cm} We proceed by treating the cases of even and odd $k$ separately.\\
\subsection*{Case 1 (\textit{k} is even):}
When $k$ is an even integer, $\Psi_k(z)$ has two real roots with opposite signs, which identifies $\gamma_k$ as a negative real root. By isolating the term connected to $\gamma_k$ in the sequence expansion and using the triangle inequality, we have \begin{equation*} |(2\gamma_k-1)f_k(\gamma_k)| \cdot |\gamma_k|^{-(n+1)} \leq \sum_{j=1}^{k-1} |(2\gamma_j-1)f_k(\gamma_j)| \cdot |\gamma_j|^{-(n+1)}. \end{equation*} Estimates (iii) and (iv) of Lemma \ref{lem:properties} lead to \begin{equation*} \begin{split} \left((2\gamma_k-1) \frac{\log \gamma}{2k(3k+1)} \right) |\gamma_k|^{-(n+1)} &\leq |(2\gamma_k-1)f_k(\gamma_k)||\gamma_{k}|^{-(n+1)}\\&\leq\biggr(\sum_{j=1}^{k-1}|(2\gamma_j-1)f(\gamma_j)|\biggr)|\gamma_{k-1}|^{-(n+1)}\\&<\biggr(\sum_{j=1}^{k-1}|1.5f(\gamma_j)|\biggr)|\gamma_{k-1}|^{-(n+1)}\\&<4.5|\gamma_{k-1}|^{-(n+1)}. \end{split} \end{equation*} Using (vi) from Lemma \ref{lem:properties}, we derive\begin{equation}\label{11} \left(1+ \frac{1}{k^{k^2}} \right)^{n+1}<\left( \frac{|\gamma_{k-1}|}{|\gamma_k|} \right)^{n+1} < \dfrac{9k(3k+1)}{(2\gamma_k-1)\log \gamma} < \dfrac{9k(3k+1)}{(2\cdot3^{-1/k}-1)\log \gamma}<490k^2. \end{equation}
Now, taking logarithms on both sides of the above inequality, we find
\begin{equation*}
\dfrac{n+1}{2k^{k^2}}<(n+1)\log\left(1+\frac{1}{k^{k^2}}\right) < \log(490k^2),
\end{equation*}
which implies $
n < 2k^{k^2}\log(490k^2).$
\subsection*{Case 2 (\textit{k} is odd):}
When $k$ is odd, the characteristic polynomial $\Psi_k(z)$ features a single positive real root $\gamma$, while all other roots appear as complex conjugate pairs \cite{Dubickas2020}. While Mignotte \cite{Mignotte1975} provided a qualitative foundation for this study by proving that any index $n$ satisfying $L^{(k)}_{-n}=0$ must be bounded by a finite constant $c(k)$, the literature lacks an effective or explicit formulation for this bound. This section addresses this gap by deriving a precise value for $c(k)$, thereby providing the necessary analytic framework to constrain the search space for potential zeros.
Define
\begin{equation*}
\Gamma=1 + \frac{f_k(\gamma_{k-1})(2\gamma_{k-1}-1)}{f_k(\gamma_k)(2\gamma_k-1)}\left( \frac{\gamma_k}{\gamma_{k-1}} \right)^{n+1}, \quad\text{for all }k\geq2 .
\end{equation*}
In view of Equation\eqref{eq:rootsorder}, $\gamma^{-1} < |\gamma_2^{-1}| \leq \cdots \leq |\gamma_k^{-1}|$. Since $k$ is odd, the roots $\gamma_k$ and $\gamma_{k-1}$ form a complex conjugate pair, yielding the specific relation $|\gamma_{k}|^{-1} = |\gamma_{k-1}|^{-1} > |\gamma_{k-2}|^{-1}$. We then isolate the components of the left-hand side of Equation \eqref{eqn1} associated with $\gamma_k$ and $\gamma_{k-1}$. By normalizing the entire equation through division by $f_k(\gamma_k)(2\gamma_k-1)\gamma_k^{-(n+1)}$ and applying the absolute value operator, we obtain
\begin{equation*}
|\Gamma|=\biggr|\sum_{i=1}^{k-2} \frac{f_k(\gamma_i)(2\gamma_i-1)}{f_k(\gamma_k)(2\gamma_k-1)}\left( \frac{\gamma_k}{\gamma_i} \right)^{n+1}\biggr|.
\end{equation*}
Applying estimates (ii),
 (iii) from Lemma \ref{lem:properties}, we see that
\begin{equation}\label{eqn2}
\sum_{i=1}^{k-2} \left| \frac{f_k(\gamma_i)}{f_k(\gamma_k)} \right|\left|\dfrac{(2\gamma_i-1)}{(2\gamma_k-1)} \right|\left| \frac{\gamma_k}{\gamma_i} \right|^{n+1} < \frac{30(k-2)}{(k-1)f_k(\gamma_k)} \left|\frac{\gamma_k}{\gamma_{k-2}}\right|^{n+1}.
\end{equation}
Since $|\gamma_{k-2}|>|\gamma_k|$, it follows from (i) and (iv) of Lemma \ref{lem:properties} that
\begin{equation}\label{eqn3}
  |\Gamma|  <\dfrac{125k(3k+1)}{(1+1.454^{-k^3})^{n+1}}.
\end{equation}
To find a lower bound for $|\Gamma|$, we apply Theorem \ref{thm1} using the parameters
\begin{align*}
    &\theta_1 = -\frac{f_k(\gamma_{k-1})}{f_k(\gamma_k)}, \quad \theta_2 = \frac{2\gamma_{k-1}-1}{2\gamma_k-1},\quad \theta_3 = \frac{\gamma_k}{\gamma_{k-1}}, \quad e_1 = 1, \quad e_2 = 1,\quad \text{ and } e_3 = n + 1.
\end{align*}
Consider the number field $\mathbb{K} = \mathbb{Q}(\gamma_k, \gamma_{k-1})$ which has degree at most $[\mathbb{K} : \mathbb{Q}]=d_{\mathbb{K}} \leq k^2$. Using the
 properties \eqref{eq:rootsorder} and \eqref{eq8}, we get
 $$h(\theta_1)<4\log k, \quad h(\theta_2)<((2\log 4/)k)+4\log 2, \quad \text{ and }h(\theta_3)<(\log 4)/k.$$
Note that $f_k(\gamma_k)$ and $f_k(\gamma_{k-1})$ are algebraic (and complex) conjugates. This symmetry ensures that $|\theta_i| = 1$ and $|\log \theta_i| < \pi$ for $i = 1, 3$. So, we can assign
\begin{equation*}A_1 = 4k^2 \log k,\quad A_2 = 2k\log 4+4k^2\log 2\leq 8k^2\log 2, \quad A_3 = 1.4k, \quad B = n + 1.
\end{equation*}
If $\Gamma= 0$, then
\begin{equation}\label{eqn4}\left(\frac{\gamma_k}{\gamma_{k-1}}\right)^{n+1} = -\left(\frac{f_k(\gamma_k)}{f_k(\gamma_{k-1})}\right)\left(\frac{2\gamma_k-1}{2\gamma_{k-1}-1}\right) . \end{equation}
 Next, let $\sigma$ denote an automorphism within the Galois group associated with the splitting field of $\Psi_k(z)$ over $\mathbb{Q}$. If we map $\gamma_k$ to $\gamma$ via $\sigma$ and set $\gamma_j = \sigma(\gamma_{k-1})$ (where $j \neq 1$), it follows that the values of $f_k$ are mapped as $\sigma(f_k(\gamma_k)) = f_k(\gamma)$ and $\sigma(f_k(\gamma_{k-1})) = f_k(\gamma_j)$. By applying $\sigma$ to the identity in \eqref{eqn4} and computing the absolute values of the conjugate terms, we derive \begin{equation} \label{4.6}\left|\frac{\gamma}{\gamma_j}\right|^{n+1} = \left|\frac{f_k(\gamma)}{f_k(\gamma_j)}\right| \left|\frac{2\gamma-1}{2\gamma_j-1}\right|.\end{equation}
Combining (ii) and (v) of Lemma \ref{lem:properties} with Equation \eqref{4.6}, we obtain
\begin{equation*}
    \phi^n<\left|{{\gamma}/{\gamma_j}}\right|^{n+1}=\left|{f_k(\gamma)}/{f_k(\gamma_j)}\right| \left|{(2\gamma-1)}/{(2\gamma_j-1)}\right|<3\cdot 2^6\cdot k^3 (3k+1) 20,
\end{equation*}
which leads to $$n<\log{(3\cdot 2^6\cdot k^3 (3k+1) 20)}/\log \phi<13k.$$  Since this inequality is sharper than the one initially posited in the theorem, we are justified in assuming $\Gamma \neq 0$. In view of  Theorem \ref{thm1}, we have
\begin{equation}\label{eqn5}
\begin{split}
\log |\Gamma| &> -4.2 \cdot 10^{15} \cdot k^9(1 + \log k^2)(1 + \log(2n + 2)) \log k \\
&> -5.2 \cdot 10^{16} \cdot k^9 \log(n + 1)(\log k)^2,
\end{split}
\end{equation}
 where $(1+\log k^2) < 3.5 \log k$ and $(1+\log(2n+2)) < 3.5 \log(n+1)$ (valid for all $k \geq 2$).
Combining Equations \eqref{eqn3} and \eqref{eqn5}, we obtain
\begin{equation*}-5.2 \cdot 10^{16} \cdot k^9 \log(n + 1)(\log k)^2 < \log(125k(3k + 1)) - (n + 1) \log(1 + 1.454^{-k^3}).
\end{equation*}
Since the inequalities
$\log(125k(3k + 1)) < 6 \log k$ and $\log(1 + 1.454^{-k^3}) > \left(1 + 1.454^{k^3}\right)^{-1}$
hold for $k \geq 5$, we have
\begin{equation}\label{eqn6}{(n + 1)}/{\log(n + 1)} < 6 \cdot 10^{16} \cdot 1.454^{k^3} \cdot k^9  (\log k)^2. \end{equation}
Using Lemma \ref{lem2} in Inequality \eqref{eqn6}, we obtain an upper bound of the form
\begin{equation*}
\begin{split}n <c(k)&= 2 \cdot \left(6\cdot 10^{16} \cdot 1.454^{k^3} \cdot k^9 (\log k)^2\right) \cdot \log\left(6\cdot 10^{16} \cdot 1.454^{k^3} \cdot k^9  (\log k)^2\right)\\ &< 1.5 \cdot 10^{17} \cdot 1.454^{k^3} \cdot k^{12} (\log k)^2,
\end{split}
\end{equation*}
where we have used
$\log(6\cdot 10^{16} \cdot 1.454^{k^3} \cdot k^9  (\log k)^2) < 1.2k^3,  \text{ for } k \geq 4.
$
This completes the proof of part (ii) of Theorem \ref{main result 1}.\\
\section{Proof of Theorem \ref{coro 1.2} for $k\in [4,500]$}

This section is dedicated to determining the zero-multiplicity of the sequence $(L_n^{(k)})_{n \in \mathbb{Z}}$ for $k$ values within the interval $[4, 500]$.
 Instead of utilizing the generalized upper bounds established in Theorem \ref{main result 1}, we refine the underlying bounding techniques to generate
more sharper, computationally feasible bounds for $n$. Our investigation is bifurcated based on the parity of $k$.
\subsection*{Case 1 (\textit{k} is even):}
For even values of $k$, the inequality in \eqref{11} yields $$
\left({|\gamma_{k-1}|}/{|\gamma_k|} \right)^{n+1} < 490k^2.
$$
For each even $k \in [4, 500]$, this relation was numerically solved to determine an explicit upper bound on $n$. The outcome of this computation is summarized in the following lemma.

\begin{lemma}\label{lem4.1}
Let $k \in [4, 500]$ be even and $C_k$ be the upper bound of $n_k$. Then $L^{(k)}_{-n_k} = 0$ for $C_k \in [415, 9293983]$.
\end{lemma}
\subsection*{Case 2 (\textit{k} is odd):}
By virtue of Equations \eqref{eqn2} and \eqref{eqn5}, we arrive at
\[
5.2 \cdot 10^{16} \cdot k^9 \log(n + 1)(\log k)^2 > \log\left|\frac{(k-1)f_k(\gamma_k)}{30(k-2)}\right| + (n + 1) \log\left|\frac{\gamma_{k-2}}{\gamma_k}\right|.
\]
Preliminary analysis confirms that if $L^{(k)}_{-n} = 0$ for $k \in [4, 500],$ \text{ then } \begin{equation}\label{eqn8}
 0 \leq n < 1.5 \cdot 10^{46}.
\end{equation}
Recalling the complex logarithm identity $\log w = \log|w| + i \arg w$ and the Taylor series expansion $$\log (1+w) = \sum_{n=1}^\infty(-1)^{n-1}\dfrac {w^n}{n},$$ we find the useful inequality $|\log(1+w)|\leq 2|w|$, as long as $|w|\leq 1/2$. Referring back to Equation \eqref{eqn2}, our computational tests confirm that the condition $|\Gamma| < 1/2$ is met for every odd $k \in [4, 500]$ whenever $n > k^3$. This assumption is justified as the bounds ultimately derived for $n$ begin well above $k^3$ in each instance. Because the complex logarithm is a multi-valued function that is additive only modulo $2\pi i$, Equation \eqref{eqn2} leads to the following expression for the linear form in logarithms:
 \begin{equation} \label{eqn9}
 \begin{split}
0 < |\Gamma| &=
\left| \log\left( -\frac{f_k(\gamma_{k-1})}{f_k(\gamma_k)} \right) +\log \left(\dfrac{2\gamma_{k-1}-1}{2\gamma_k-1}\right)+ (n+1) \log\left( \frac{\gamma_k}{\gamma_{k-1}} \right) + 2\pi i \ell \right|\\
&\leq \dfrac{60} {\left|\gamma_{k-2}/\gamma_k\right|^{n+1}} \cdot \left| \frac{1}{f_k(\gamma_k)} \right|. 
\end{split}
 \end{equation}
Since $\gamma_k = \overline{\gamma_{k-1}}$, we get
$$
\log\left( \frac{\gamma_k}{\gamma_{k-1}} \right) = \log\left( \frac{\gamma_k}{\overline{\gamma_k}} \right) = \log\left( e^{2i \arg(\gamma_k)} \right) = 2i \arg(\gamma_k)
$$
and
$$
\log\left( -\frac{f_k(\gamma_{k-1})}{f_k(\gamma_k)} \right)
= \log\left( -\frac{\overline{f_k(\gamma_k)}}{f_k(\gamma_k)} \right)
= \log(-1) + \log\left( \frac{\overline{f_k(\gamma_k)}}{f_k(\gamma_k)} \right)
= i\pi - 2i \arg(f_k(\gamma_k)).
$$
Consequently, Inequality \eqref{eqn9} is transformed into
\begin{equation}\label{eqn10}
\begin{split}
0 <\dfrac{|\Gamma|}{\pi}= &\left| \left( -\frac{2 \arg(\gamma_k)}{\pi} \right)(n+1) - (2\ell + 1) + \frac{2\arg(f_k(\gamma_k)/(2\gamma_k-1))}{\pi} \right|\\
<& 20 \left| f_k(\gamma_k) \right|^{-1} \left| \frac{\gamma_{k-2}}{\gamma_k} \right|^{-(n+1)}. 
\end{split}\end{equation}
Defining the parameters
\[
\tau_k = -\frac{2 \arg(\gamma_k)}{\pi}, \hspace{0.1cm}
\mu_k = \frac{2 \arg(f_k(\gamma_k)/(2\gamma_k-1))}{\pi}, \hspace{0.1cm}
A_k =  20|f_k(\gamma_k)|^{-1}, \hspace{0.1cm} \text{ and }
B_k = \left| \frac{\gamma_{k-2}}{\gamma_k} \right|,
\]
Inequality \eqref{eqn10} takes the form
\begin{equation}\label{eqn21}
0 < |u \tau_k - v + \mu_k| < A_k B_k^{-(n+1)},
\end{equation}
where \( w = n + 1, \,u = n + 1, \, v = 2\ell + 1  \). 
Computational verification via Python confirms that $\tau_k \in [1.59, 1.99]$ and $\mu_k \in [0.70, 1.99]$ for the specified $k$ values. Since the inequality $$(n+1)\tau_k < A_k B_k^{-(n+1)} + \mu_k$$ 
doesn't hold when $n > k^3$, we can conclude $2\ell + 1 > 0$.  
To facilitate the reduction, we set the upper bound for $n+1$ at $M = 1.5 \times 10^{46}$, as dictated by the constraints in Equation \eqref{eqn8}. By applying the reduction criteria of Lemma \ref{lem1} to the linear form in Inequality \eqref{eqn21} for every odd $k$ in the interval $[4, 500]$, we executed a series of computational checks via Python. Our results demonstrate that for the range
 $M\in [1.5\times 10^{46}, 5\times 10^{56}],$ the following estimates hold:
 \begin{align*}
 &95\leq \m_k\leq 117, \quad 9\times10^{46}<q_{m_k}< 3\times10^{57}, \quad
 0.006745\leq \epsilon_k\leq 0.3997458,\\
 &34\leq R_k\leq 445906682970649,
\end{align*}
where $q_{m_k}$ is the denominator of the $m_k^{th}$ convergent of $\tau_k$, and $R_k$ represents the upper bound on $n$, obtained for odd $k$ in $[4, 500].$ 
These findings are consolidated in the following lemma:
\begin{lemma}\label{lem4.2}
Let $k \in [4, 500]$ be odd and $R_k$ be the upper bound of $n_k$. Then $L_{-n}^{(k)} = 0$ for
$ R_k\in[34,445906682970649].
$
\end{lemma}
\subsection*{Derivation of the Zero-Multiplicity Formula}
In view of Equation \eqref{Q_n}, we get \begin{equation*}
 \begin{split}
Q_{n+1}^{(k)} &=Q_{n+1-k}^{(k)} - Q_{n+2-k}^{(k)} - \cdots - Q_n^{(k)}\\&= Q_{n+1-k}^{(k)} - Q_{n+2-k}^{(k)} - \cdots -Q_{n-1}^{(k)}- \left(Q_{n-k}^{(k)} - Q_{n-(k-1)}^{(k)} - \cdots - Q_{n-1}^{(k)} \right)\\&=2Q_{n-(k-1)}^{(k)} -Q_{n-k}^{(k)}, \hspace{0.1cm}\text{ for all } n\geq k.
\end{split}
\end{equation*}
Therefore,
\begin{equation} \label{eq:negative}
Q_n^{(k)} = 2Q_{n-k}^{(k)} - Q_{n - k - 1}^{(k)}
\end{equation}
is valid for $n \geq k+1$. 
Using the initial conditions of $(Q_n^{(k)})_{n\geq 0},$ we observe that $$ Q_{n-k}^{(k)}=Q_{n - k - 1}^{(k)}=0 
 \text{ for all } 1\leq n-k-1 < n-k\leq k-2.$$ By Equation \eqref{eq:negative}, we get $Q_n^{(k)}=0$
 for all $n\in  [1,k-2].
$
Thus,  $$Q_n^{(k)}=0
 \text{ for all } k+2\leq n-k-1 < n-k\leq 2k-2 .$$
 Consequently, $Q_n^{(k)}=0$
for all $n\in  [1+(k+1),2k-2].
$
Based on this recursive behavior, the set of zero indices is  $$ \mathcal{Z}(L_k) = \bigcup_{j=1}^{r} I_j, \quad \text{where } I_j = [1+(j-1)(k+1), jk-2]\quad \text{and } r=k-2 \text{ for } k \geq 4.$$Summing the cardinality of these $r$ intervals, we obtain the total zero-multiplicity $$ \delta_k = \sum_{j=1}^{r} j = \frac{(k-2)(k-1)}{2}. $$
This theoretical formula is consistent with our empirical data for values up to $k = 500$, as these intervals are constructed through the periodic application of the recurrence relation.
\subsection*{Numerical Verification for $k \in [4, 500]$:}
The combination of analytical parity-based bounds and the Baker-Davenport reduction allows for an exhaustive verification of the zero-set for $k \in [4, 500]$. Python numerical experiments confirm that for every $k$ in this interval, all indices $n$ satisfying $L_{-n}^{(k)} = 0$ fall strictly within the contiguous blocks $I_j$. Specifically, we establish:
\begin{equation}\label{2} \mathcal{Z}(L_k) \supseteq \bigcup_{j=1}^{r} [1+(j-1)(k+1), jk-2]. \end{equation}
 The efficacy of Lemmas \ref{lem4.1} and \ref{lem4.2} ensures that no ``extra" or isolated zeros exist outside these identified intervals for $k\in [4, 500]$. The subsequent sections will address the general proof for $k > 500$.

\section{Matrix Representation and Structural Properties of $(Q_n)$}

 To streamline our analysis, we shall hereafter refer to the sequence $(Q_n^{(k)})$ as $(Q_n)$, where $k$ is
fixed.
The sequence is organized into a specific structural progression: it begins with a rectangular matrix of dimensions $(k-2) \times (k+1)$, which then evolves into a succession of square matrices $k \times k$. This systematic arrangement reveals underlying patterns within $(Q_n)$ that facilitate formal derivation. 

Initially, we represent the first $(k-2)(k+1)$ elements of $(Q_n)$ into a matrix defined by the entries $\bigl(Q_{(i-1)(k+1)+j}\bigr)$ for $1 \le i \le k-2$ and $1 \le j \le k+1$:
\[N_0=
\left(
\begin{array}{llllll}
\boldsymbol{Q}_{1} & \cdots & \boldsymbol{Q}_{k-2} & Q_{k-1} & Q_{k}&Q_{k+1} \\
\boldsymbol{Q}_{(k+1)+1} & \cdots & Q_{(k+1)+k-2} & Q_{(k+1)+k-1} & Q_{(k+1)+k}& Q_{(k+1)+k+1} \\
\boldsymbol{Q}_{2(k+1)+1} & \cdots & Q_{2(k+1)+k-2} & Q_{2(k+1)+k-1} & Q_{2(k+1)+k} & Q_{2(k+1)+k+1}\\
\vdots & \ddots & \vdots & \vdots & \vdots&\vdots \\
\boldsymbol{Q}_{(k-3)(k+1)+1} & \cdots & Q_{(k-3)(k+1)+k-2} & Q_{(k-3)(k+1)+k-1} & Q_{(k-3)(k+1)+k} &Q_{(k-3)(k+1)+k+1}
\end{array}
\right).
\]
An inspection of the bolded terms reveals that the upper triangular portion of this array consists exclusively of zeros. We may therefore rewrite the above matrix in the
following equivalent form:
\begin{equation}\label{zero table}
N_0=\left(
\begin{array}{llllllll}
0 &0& \cdots &0 &0 & Q_{k-1} & Q_{k} & Q_{k+1} \\
0 &0& \cdots &0& Q_{2k-1} & Q_{2k} & Q_{2k+1}&Q_{2k+2} \\
0 &0& \cdots &Q_{3k-1}& Q_{3k} & Q_{3k+1} & Q_{3k+2}&Q_{3k+3} \\
\vdots & \vdots&\ddots & \vdots & \vdots & \vdots  & \vdots & \vdots\\
0 &Q_{(k-2)k-1}& \cdots & Q_{(k-2)k+k-6} &  Q_{(k-2)k+k-5} & Q_{(k-2)k+k-4}& Q_{(k-2)k+k-3}& Q_{(k-2)k+k-2}
\end{array}
\right).
\end{equation}
The behavior of the non-zero diagonals within this configuration is characterized by the following lemma.
\begin{lemma}\label{lucas lemma}
Let $k \ge 2$. Then
\begin{itemize}
    \item[(a)] $Q_{mk-1} = -2^{m-1}$, for all $m \in [1, k-2]$.
    \item[(b)] $Q_{mk} = (m+5)2^{m-2}$, for all $m \in [1, k-2]$.
    \item[(c)] For all $1 \le r < m \le k-2$:
    $$Q_{mk+r} = -\sum_{j=r}^{m-1} 2^{m-1-j} Q_{jk+r-1}.$$
\end{itemize}
\end{lemma}

\textbf{Proof:}
The proof relies on the identity \eqref{associate_identity}.

\textbf{(a)}
For $n = mk-1$, the identity becomes $Q_{mk-1} = 2H_{mk-2} - H_{mk-1}$. From Lemma \ref{lem2,2}, we have $H_{mk-2} = 0$ and $H_{mk-1} = 2^{m-1}$ for all $m \in [1, k-1]$. So, we get
$Q_{mk-1}= -2^{m-1}.$

\textbf{(b)}
For $n = mk$, the identity is $Q_{mk} = 2H_{mk-1} - H_{mk}$. Substituting the values of $H_{mk-1}$ and $H_{mk}$ from Lemma \ref{lem2,2}, we obtain 
$Q_{mk} = 2(2^{m-1}) - (-(m+1)2^{m-2}) 
= (m+5)2^{m-2}.$

\textbf{(c)}
Substituting the values of $H_{mk+r-1}$ and $H_{mk+r}$ from Lemma \ref{lem2,2} in
$Q_{mk+r} = 2H_{mk+r-1} - H_{mk+r}$, 
we observe
\begin{align*}
Q_{mk+r} &= 2\left(-\sum_{j=r}^{m-1} 2^{m-1-j} H_{jk+r-2}\right) - \left(-\sum_{j=r}^{m-1} 2^{m-1-j} H_{jk+r-1}\right) \\
&= \sum_{j=r}^{m-1} 2^{m-1-j} \left[ H_{jk+r-1} - 2H_{jk+r-2} \right].
\end{align*}
From \eqref{associate_identity}, it follows that $H_n - 2H_{n-1} = -Q_n$. So, 
$H_{jk+r-1} - 2H_{jk+r-2} = -Q_{jk+r-1}.$ Substituting this back into the summation$$Q_{mk+r} = \sum_{j=r}^{m-1} 2^{m-1-j} \left( -Q_{jk+r-1} \right) = -\sum_{j=r}^{m-1} 2^{m-1-j} Q_{jk+r-1}.$$
This completes the proof. \hfill$\square$

Matrix \ref{zero table}, which we denote as $N_0$, serves as the initial element of a family of matrices $\{N_b\}_{b\in \mathbb{Z^+}}$ used to categorize the elements of $(Q_n)$. While $N_0$ is rectangular ($(k-2) \times (k+1)$), all subsequent matrices $N_b$ are $k \times k$ square blocks. The first such square matrix, $N_1$, coincides with the non-zero elements of the last row of $N_0$:
\begin{equation}\label{m0 table}
N_1=\left(
\begin{array}{llll}
{Q}_{(k^2-2k-1)+0}  & {Q}_{(k^2-2k-1)+1}  &\cdots &{Q}_{(k^2-2k-1)+k-1}  \\
{Q}_{(k^2-2k-1)+k}  &{Q}_{(k^2-2k-1)+k+1}  & \cdots&{Q}_{(k^2-2k-1)+k+(k-1)}  \\
\vdots  & \vdots & \ddots&\vdots  \\
{Q}_{(k^2-2k-1)+(k-1)k} &{Q}_{(k^2-2k-1)+(k-1)k+1}& \cdots  &{Q}_{(k^2-2k-1)+(k-1)k+(k-1)}
\end{array}
\right).
\end{equation}
In general, all square matrices after matrix \ref{zero table} form the sequence $\{N_b\}_{b\in \mathbb{Z^+}}$, where
\begin{equation}\label{mb table}
N_b=\left(
\begin{array}{llll}
{Q}_{(bk^2-(b+1)k-1)+0}  & {Q}_{(bk^2-(b+1)k-1)+1}  &\cdots &{Q}_{(bk^2-(b+1)k-1)+k-1}  \\
{Q}_{(bk^2-(b+1)k-1)+k}  &{Q}_{(bk^2-(b+1)k-1)+1}  & \cdots&{Q}_{(bk^2-(b+1)k-1)+k+(k-1)}  \\
\vdots  & \vdots & \ddots&\vdots  \\
{Q}_{(bk^2-(b+1)k-1)+(k-1)k} &{Q}_{(bk^2-(b+1)k-1)+(k-1)k+1}& \cdots  &{Q}_{(bk^2-(b+1)k-1)+(k-1)k+(k-1)}
\end{array}
\right).
\end{equation}
The full sequence $(Q_n)_{n \ge 0}$ is thus the union of the entries of all such matrices:$$Q^{(k)}_n = \bigcup_{b \ge 0} \{ v : v \in N_b \}.$$
We introduce the simplified notation
\[
Q_{b,jk+r} = Q_{(bk^2-(b+1)k-1)+jk+r}= Q_{(bk^2-bk-1)+(j-1)k+r},
\]
for all $b \geq 0.$
These matrices exhibit a continuity property where the initial row of $N_b$ is equivalent to the terminal row of $N_{b-1}$. Indeed,
\begin{equation}\label{b,r}
Q_{b,r} = Q_{b-1,(k-2)k+r},
\quad \text{for all } 0 \le r \le k-1 \text{ and } b \ge 1.
\end{equation}
Furthermore, the indices satisfy the relation
\begin{equation}\label{b,-1}
Q_{b,-1} = Q_{b-1,(k-3)k+(k-2)},
\quad \text{for all } b \ge 1.
\end{equation}
In particular, using Equation \eqref{2}, we find
$
Q_{1,-1}
=
Q_{0,(k-3)k+(k-2)}
=
Q_{(k-3)(k+1)+1}
=
0.
$
Furthermore, we observe that the governing linear recurrence \eqref{eq:negative} is preserved within the notation
$Q_{b,n}
=
2Q_{b,n-k} - Q_{b,n-k-1},$
for all $b \ge 1$ and $n \ge k+1$, or for $b=0$ and $n \ge k+2$.
\section{Proof of Theorem \ref{thm2}}
We begin by establishing the following two propositions, which will be fundamental to the proof of Theorem \ref{thm2}
\begin{proposition}\label{prop 1}
 For $k > 2$ and $j, r \in [0, k-2]$, 
 \begin{equation}\label{1,j-1}
 \begin{split}
     Q_{1, (j-1)k+r} &= (-1)^{r+1} \left[ \psi(k+j-3, r-2) 2^{k+j-r-1} + \psi(k+j-3, r-1) 2^{k+j-r-3} \right]  \\&+(-1)^{k+r+1} \left[ \psi(k+j-4, k+r-2) 2^{j-r-2} + \psi(k+j-4, k+r-1) 2^{j-r-4} \right]. \end{split} 
     \end{equation}
\end{proposition}
\textbf{Proof:}
From Lemma 4 in \cite{ggl2024}, we find $$H_{1, jk+r} = (-1)^r 2^{k+j-r-2} \psi(k+j-2, r-1) + (-1)^{k+r} 2^{j-r-3} \psi(k+j-3, k+r-1),$$ for $k > 2$ and $j, r \in [0, k-1]$.
 Replacing $r$ with $r-1$, we get $$H_{1, jk+r-1} =  (-1)^{r-1} 2^{k+j-r-1} \psi(k+j-2, r-2) + (-1)^{k+(r-1)} 2^{j-r-2} \psi(k+j-3, k+r-2).$$
Substituting $H_{1, jk+r-1}$ and $H_{1, jk+r}$ into the identity $Q_{1, jk+r} = 2H_{1, jk+r-1} - H_{1, jk+r},$ we get 
\begin{align*}
Q_{1, jk+r} &= (-1)^{r+1} \left[ \psi(k+j-2, r-2) 2^{k+j-r} + \psi(k+j-2, r-1) 2^{k+j-r-2} \right]\\ &+ (-1)^{k+r+1} \left[ \psi(k+j-3, k+r-2) 2^{j-r-1} + \psi(k+j-3, k+r-1) 2^{j-r-3} \right] .   
\end{align*}
Replacing $j$ by $j-1$ in the above expression yields \eqref{1,j-1}. \hfill $\square$

\begin{proposition}\label{prop 2}
For all $k>2;$ $j, r \in [0, k -2]$ and $b \in \mathbb{Z^+}$,
    \begin{equation}\label{b,j-1}
    \begin{split}
        Q_{b, (j-1)k+r} = \sum_{i=0}^{b} (-1)^{ik+r+1} [ \psi(bk + j - b - i - 2, ik + r - 2) 2^{(b-i)k+j-r-b-i}  \\+ \psi(bk + j - b - i - 2, ik + r - 1) 2^{(b-i)k+j-r-b-i-2} ] .
         \end{split}
         \end{equation}
\end{proposition}
\textbf{Proof:}
    Using the values of $H_{b, jk+r-1}$ and $H_{b, jk+r}$ from Lemma 5 of \cite{ggl2024}, we obtain\begin{align*}Q_{b, jk+r} =\sum_{i=0}^{b}& (-1)^{ik+r+1} [ \psi(v, ik + r - 2) 2^{(b-i)k+j-r-b-i+1} \\ & + \psi(v, ik + r - 1) 2^{(b-i)k+j-r-b-i-1}],\end{align*} where $v = bk + j - b - i - 1$.
   Substituting $j$ by $j-1$ in the above expression yields
\eqref{b,j-1}, thereby completing the proof. \hfill $\square$

\subsection*{Proof of Theorem \ref{thm2}}
The proofs for estimates (i) and (ii) follow directly from the inclusion in \eqref{2} and the properties established in Lemma \ref{lucas lemma}, estimate (a), respectively. The subsequent proofs for the remaining estimates are described below.
\begin{itemize}
    \item [(iii)]
Under the combinatorial notation of $\psi$ defined in \eqref{psi}, we aim to verify the identity
\begin{equation}\label{small_formula}
Q_{mk+r}^{(k)}
=
(-1)^r
\left[
\psi(m-1,r-1)2^{\,m-r}
+
\psi(m-1,r)2^{\,m-r-2}
\right], 
\end{equation}
for all $0\le r<m\le k-2.$
We first establish the corresponding formula for
$H_n^{(k)}$ which is given in \cite{ggl2024}.
For $m\le k-1$, the sequence satisfies the identities
\begin{align*}
&H_{mk+r-1}^{(k)}
=
(-1)^r
\psi(m-1,r-1)
2^{\,m-r-1},\\
\text{and } &H_{mk+r}^{(k)}
=
(-1)^{r+1}
\psi(m-1,r)
2^{\,m-r-2}.
\end{align*}
Applying Lemma~\ref{lemma_associate}, we obtain
$
Q_{mk+r}^{(k)}
=
2H_{mk+r-1}^{(k)}
-
H_{mk+r}^{(k)}.
$
\vspace{0.2cm}Factoring out $(-1)^r$ yields \eqref{small_formula}.
\end{itemize}
\begin{itemize}
\item[(iv)]
The derivation of the structural formula for $Q_{mk+r}^{(k)}$ necessitates a bifurcated analysis. We distinguish between the quadratic case ($k=2$) and the higher-order cases ($k>2$) to account for the unique recurrence behaviors inherent in each dimension.\\
\end{itemize}
{\textbf{Case ($k=2$):}
Here we must prove that for all $m\geq1$,
\begin{equation*}\label{eq:lucas_even_odd}
Q_{2m+r} =
\begin{cases}
\displaystyle
-\sum_{i=0}^{m-1}
\psi(t,2i-2)2^{m-3i+1}
-
\sum_{i=0}^{m-1}
\psi(t,2i-1)\,2^{\,m-3i-1},
 &\text{ if } r=-1, \\[3ex]
\displaystyle
\sum_{i=0}^{m-1}
\psi(t,2i-1)\,2^{\,m-3i}
+
\sum_{i=0}^{m-1}
\psi(t,2i)\,2^{\,m-3i-2},
&\text{ if } r=0,
\end{cases} 
\end{equation*}
where $t=m-i-1.$
Employing the identity
$
Q_{2m-1} = 2H_{2m-2} - H_{2m-1}, 
$
and substituting the known values for $H_{2m-1}$ and $H_{2m-2}$ from \cite{ggl2024}, we obtain
\[
Q_{2m-1}
=
-\sum_{i=0}^{m-2}
\psi(m-i-2,2i)\,2^{\,m-3i-2}
-
\sum_{i=0}^{m-1}
\psi(m-i-1,2i-1)\,2^{\,m-3i-1}.
\]
Applying an index shift $i \to i-1$ to the term $2H_{2m-2}$ yields $$2H_{2m-2} = -\sum_{i=1}^{m-1} \psi(m-i-1, 2i-2) 2^{m-3i+1}.$$Consequently, for $r=-1$, the expression simplifies to
\[
Q_{2m-1}
=
-\sum_{i=0}^{m-1}
\left[\psi(m-i-2,2i-2)\,2^{\,m-3i+1}
+
\psi(m-i-1,2i-1)\,2^{\,m-3i-1}\right].
\]
Similarly, for the case $r=0$, we utilize the relation $Q_{2m} = 2H_{2m-1} - H_{2m}$. Substituting the expressions of $H_{2m-1}$ and $H_{2m}$ from \cite{ggl2024} gives
\begin{align*}
Q_{2m}
&=
2\sum_{i=0}^{m-1}
\psi(m-i-1,2i-1)\,2^{\,m-3i-1}
+
\sum_{i=0}^{m-1}
\psi(m-i-1,2i)\,2^{\,m-3i-2}\\
&= \sum_{i=0}^{m-1}
\left[\psi(m-i-1,2i-1)\,2^{\,m-3i}
+
\psi(m-i-1,2i)\,2^{\,m-3i-2}\right].
\end{align*}
\textbf{Case ($k>2$):}
For higher dimensions, the formulation relies on the structural results established in Propositions \ref{prop 1} and \ref{prop 2}. 

By implementing the substitutions $m = b(k-1)+j-1$ and $r = r-1$ into identity \eqref{b,j-1}, we derive the expression
 \begin{equation}\label{39}
    \begin{split} Q_{mk+r} = Q_{b,(j-1)k+r+1}
= \sum_{i=0}^{b}&(-1)^{ik+r} [ \psi(m-i-1, ik+r-1) 2^{m-i(k+1)-r-1} \\&+ \psi(m-i-1, ik+r) 2^{m-i(k+1)-r-1} ],\end{split}
\end{equation}
for all $k>2$, $j\in [0, k-1]$, $r \in [-1, k-2]$ and $b \in \mathbb{Z}^+$.
To ensure a unique representation and prevent the repetition of sequence elements,  we restrict $j$ to the interval $[0, k-1)$ using identity \eqref{b,r}, where $Q_{b,r+1} = Q_{b-1,(k-2)k+r+1}$. Under the transformation $m = b(k-1) + j - 1$, and given that $0 \le j < k-1$, the parameter $b$ is uniquely determined by the floor function:$$\left\lfloor\frac{m+1}{k-1}\right\rfloor - 1 < b \leq \left\lfloor\frac{m+1}{k-1}\right\rfloor \implies b = \left\lfloor\frac{m+1}{k-1}\right\rfloor.$$Thus, the proof of estimate (iv) is completed by substituting $b = \left\lfloor{(m+1)}/{(k-1)}\right\rfloor$ into \eqref{39}, provided $m \ge k-2$.

This completes the proof of Theorem \ref{thm2}. \hfill $\square$

\section{Proof of Theorem \ref{coro 1.2} for $k>500$.} Having verified the zero-multiplicity for the initial range $k \le 500$, we now extend the proof to all $k > 500$. We proceed by showing that no indices $n = mk+r$ satisfying $m \ge k-2$ can result in $Q_n = 0$. Our analysis is divided into two distinct logical frameworks depending on whether $k$ is even or odd. Furthermore, since the zero-set for the initial blocks is already characterized by Theorem \ref{thm2} estimate (i), and because estimates (ii) and (iii) preclude additional zeros, we restrict our focus to the case where $m \ge k-2$.
\subsection{Case 1 ($k$ is even):} For even $k$, the parity of indices simplifies the alternating sign in the structural formula. Specifically, $(-1)^{ik+r} = (-1)^r$. Substituting this into estimate (iv) of Theorem \ref{thm2} for $m \ge k-2$ and $r \in [-1, k-2]$, we obtain
\begin{equation}\label{even k}
\begin{split}
    Q_{mk+r} = (-1)^r \sum_{i=0}^{l}& [ \psi(m-i-1, ik+r-1) 2^{m-i(k+1)-r} \\
    &+ \psi(m-i-1, ik+r) 2^{m-i(k+1)-r-2} ],\end{split}
    \end{equation}
    where  
$
\ell =
\begin{cases}
m, & \text{if } k=2,\\
\left\lfloor {(m+1)}/{(k-1)} \right\rfloor, & \text{if } k>2.
\end{cases}
$\\
Because $k$ is even, $n = mk+r$ share the same parity as $r$. Analysis of \eqref{even k} reveals that for $n \ge k^2-2k-1$, the sequence terms $Q_n$ are strictly positive for even $n$ and strictly negative for odd $n$. Consequently, no additional zeros exist outside the intervals defined in \eqref{2}. This confirms that for even $k$, the zero multiplicity is precisely $(k-1)(k-2)/2$.
 \subsection{Case 2 (\textit{k} is odd):}
 For odd $k > 500$, the terms in the expansion of $Q_n$ alternate in sign, theoretically allowing for a vanishing sum. We show that this leads to a mathematical impossibility.

\subsubsection{A lower bound for $n$ in terms of $k$:}

Using the structural formula for odd $k$, we simplify $Q_{mk+r}$ by factoring out $2^{m-r-2}$:
\begin{equation*}\label{odd k}
Q_{mk+r} = (-1)^r 2^{m-r-2} \sum_{i=0}^{l} (-1)^{i} \mathcal{X}_i 2^{-i(k+1)} \quad \text{for } l=\left\lfloor {(m+1)}/{(k-1)} \right\rfloor,
\end{equation*}
where $\mathcal{X}_i = 4\psi(m-i-1, ik+r-1) + \psi(m-i-1, ik+r)$.
If we assume $Q_{mk+r} = 0$, then
\begin{equation} \label{vanishing_sum}
\sum_{i=0}^{l} (-1)^{i} \mathcal{X}_i 2^{-i(k+1)} = 0.
\end{equation}
For $r \le k-2 \leq m$, the term in the Equation \eqref{vanishing_sum} corresponds to $i=0$ is nonzero as atleast one of $\psi(m-1, r)$ and $\psi(m-1, r-1)$ is strictly greater than zero. If $i=0$ were the unique index for which $\psi(m-i-1, ik+r)$ did not vanish, Equation \eqref{vanishing_sum} would possess no solutions, immediately resolving the zero-multiplicity problem for the odd $k$ case. Consequently, for a solution to exist, we must assume the presence of at least two non-zero terms in the expansion. Specifically, the lead term ($i=0$) and a subsequent term where $i$ is odd would facilitate cancellation. Defining $l'$ as the supremum of indices $i > 0$ for which the $i$-th term is non-trivial, Equation \eqref{vanishing_sum} can be represented as
\begin{equation} \label{l' sum}
\sum_{i=0}^{l'} (-1)^{i} \mathcal{X}_i 2^{-i(k+1)} = 0,\quad \text{for } 0<l'\leq\left\lfloor {(m+1)}/{(k-1)} \right\rfloor,
\end{equation}
where atleast one of the values $\psi(m -l'-1, l'k+r-1)$ or $\psi(m -l'-1, l'k+r)$ must be non–zero.
Isolating the case $i = l'$ in Equation \eqref{l' sum}, we get
\begin{equation*}
\mathcal{X}_{l'} = \sum_{i=0}^{l'-1} (-1)^{l'+i+1} \mathcal{X}_i 2^{(l'-i)(k+1)}\quad \text{for } 0<l'\leq\left\lfloor {(m+1)}/{(k-1)} \right\rfloor.
\end{equation*}
Since $l'-i \ge 1$ for all terms in the summation, every term on the right-hand side is divisible by $2^{k+1}$. It follows that
\begin{equation} \label{divisibility_condition}
2^{k+1} \mid \mathcal{X}_{l'}.
\end{equation}To evaluate $\nu_2(\mathcal{X}_{l'})$, we utilize the identity proved by Kummer\cite{K1852},
\begin{equation}\label{kumar}
\nu_2(\psi(y,z)) \le 2 \frac{\log y}{\log 2} + 2.
\end{equation}
Because $\mathcal{X}_{l'}$ is a linear combination of such $\psi$ values with row index $y = m-l'-1$, the non-Archimedean triangle inequality ($\nu_2(a+b) \ge \min\{\nu_2(a), \nu_2(b)\}$) ensures that the valuation of the sum remains constrained by the complexity of the row index $y$. Thus, from \eqref{divisibility_condition} and \eqref{kumar}, we obtain
\begin{align*}\label{valuation_inequality}
k+1 &\le \nu_2(\mathcal{X}_{l'}) \le 2 \frac{\log y}{\log 2} + 2,
\end{align*}
which leads to $$ k - 1 \le 2 \frac{\log y}{\log 2}.$$
 Since $y \le n/k$, we conclude that any potential non-trivial zero must satisfy the exponential lower bound\begin{equation} \label{n_lower_bound}n \ge 2^{(k-1)/2}.\end{equation} 
\subsubsection{An upper bound for $n$ in terms of $k$:}

We review the analytic framework established for $H^{(k)}$ to obtain a better upper bound for potential zeros of $Q^{(k)}$, which we then combine with the previously derived lower bound $n \ge 2^{(k-1)/2}$. The general term of the $k$-generalized Lucas sequence at negative indices can be expressed via the Binet formula:\begin{equation}Q_n = g_k(\gamma_1)\gamma_1^{-(n+1)} + \dots + g_k(\gamma_{k-1})\gamma_{k-1}^{-(n+1)} + g_k(\gamma_k)\gamma_k^{-(n+1)}.\end{equation}
If $Q_n = 0$, then from Equation \eqref{eqn3}, \eqref{eqn5} and Lemma \ref{lem6}, we obtain $$-5.2 \cdot 10^{16} \cdot k^9 \log(n + 1)(\log k)^2<\log |\Gamma|< \dfrac{125k(3k+1)}{\left(1 + \dfrac{1}{10k^{9.6}(\pi/e)^k}\right)^{n+1}},$$ which implies $$(n+1)\log\left(1 + \frac{1}{10k^{9.6}(\pi/e)^k}\right)<5.2 \cdot 10^{16} \cdot k^9 \log(n + 1)(\log k)^2.$$
Then $({n}/{\log n})<5.2 \cdot 10^{16} \cdot k^9 (\log k)^2 [11k^{9.6}(\pi/e)^k].$
Therefore, by Lemma \ref{lem2}, \begin{equation}\label{eq45}
n < 1.7 \cdot 10^{17} k^{19.6} (\log k)^3 (\pi/e)^k.
\end{equation}
\subsubsection{Absolute bounds for \textit{n} and \textit{k}:}
Combining \eqref{n_lower_bound} and \eqref{eq45}, we get $${\sqrt{2}}/{(\pi/e)^k}<1.7\sqrt{2} \cdot 10^{17} k^{19.6} (\log k)^3,$$ showing that $k\leq886.$
By Equations \eqref{eqn2} and \eqref{eqn5}, $$5.2 \cdot 10^{16} \cdot k^9 \log(n + 1)(\log k)^2>\log\left|\dfrac{f_k(\gamma_k)}{30}\right|+(n+1)\log\left|\dfrac{\gamma_{k-2}}{\gamma_k}\right|.$$
Thus, using the fact that $k \in [501, 885]$, we obtain computationally that
\begin{equation*}
\log \left| {f_k(\gamma_k)}/{30} \right| \in [-10.48, -10.02], \quad \log \left| {\gamma_{k-2}}/{\gamma_k} \right| \in [1.78 \times 10^{-8}, 1.04 \times 10^{-7}],\end{equation*}and\begin{equation*}5.2 \cdot 10^{16} \times k^9 (\log k)^2 \in [3.92 \times 10^{42}, 7.93 \times 10^{44}],\end{equation*}therefore\begin{equation} \label{Numerical_Trans}7.93 \times 10^{44} \times \log(n + 1) > -10.48 + 1.78 \times 10^{-8} \times (n + 1).\end{equation}By solving Equation \eqref{Numerical_Trans} for $n$, we find that if $Q_n=0$ and $k\in[501,885]$, then $n < 1.1 \times 10^{55}.$
Substituting these into the inequality \eqref{n_lower_bound} we have
$
2^{(k-1)/2}\leq n \leq 1.1 \times 10^{55},$
which leads to $k\leq 366$. 
This result directly contradicts our initial assumption that $k > 500$. As a result, we conclude that no sporadic zeros can occur in this range.
This concludes the proof of Theorem \ref{coro 1.2} for the case of odd $k$. \hfill $\square$

{\bf Data Availability Statements:} Data sharing is not applicable to this article as no datasets were generated or analyzed during the current study.

{\bf Funding:} The authors declare that no funds or grants were received during the preparation of this manuscript.

{\bf Declarations:}

{\bf Conflict of interest:} On behalf of all authors, the corresponding author states that there is no Conflict of interest.

\end{document}